\documentclass[11pt,a4paper]{article}
\usepackage{amsmath,amsthm,amsfonts,amssymb,enumerate,calc}
\usepackage[colorlinks=true,citecolor=black,linkcolor=black,urlcolor=blue]{hyperref}
\usepackage[numbers,sort&compress]{natbib}
\usepackage[capitalise]{cleveref}
\usepackage[utf8]{inputenc}
\usepackage[lmargin=30mm,rmargin=30mm,bmargin=30mm,tmargin=30mm]{geometry}

\title{\bf Average degree conditions forcing a minor}
\author{Daniel J. Harvey \qquad David R. Wood}

\newcommand{\msn}[1]{MR:\,\href{http://www.ams.org/mathscinet-getitem?mr=MR#1}{#1}}
\newcommand{\doi}[1]{doi:\,\href{http://dx.doi.org/#1}{#1}}
\newtheorem{theorem}{Theorem}
\newtheorem{lemma}[theorem]{Lemma}
\newtheorem{proposition}[theorem]{Proposition}
\newcommand{\ceil}[1]{\lceil{#1}\rceil}
\newcommand{\floor}[1]{\lfloor{#1}\rfloor}

\begin{document}
\maketitle

\begin{abstract}
\citeauthor{fast1} first proved that high average degree forces a given graph as a minor. Often motivated by Hadwiger's Conjecture, much research has focused on the average degree required to force a complete graph as a minor. Subsequently, various authors have consider the average degree required to force an arbitrary graph $H$ as a minor. Here, we strengthen (under certain conditions) a recent result by \citeauthor{superfast5}, giving better bounds on the average degree required to force an $H$-minor when $H$ is a sparse graph with many high degree vertices. This solves an open problem of \citeauthor{superfast5}, and also generalises (to within a constant factor) known results when $H$ is an unbalanced complete bipartite graph.
\end{abstract}

\section{Introduction}

\citet{fast1,fast2} first proved that high average degree forces a given graph as a minor\footnote{A graph $H$ is a \emph{minor} of a graph $G$ if a graph isomorphic to $H$ can be constructed from $G$ by vertex deletion, edge deletion and edge contraction.}. In particular, \citet{fast1,fast2} proved that the following function is well-defined, where $d(G)$ denotes the average degree of a graph $G$:  $$f(H) := \inf\{D \in \mathbb{R} : \text{ every graph }G\text{ with }d(G)\geq D\text{ contains an $H$-minor}\}.$$

Often motivated by Hadwiger's Conjecture (see \citep{Seymour-HC}), much research has focused on $f(K_t)$ where $K_t$ is the complete graph on $t$ vertices. For $3\leq t\leq 9$, exact bounds on the number of edges due to \citet{fast2}, \citet{jorg}, and \citet{SongT} imply that $f(K_t)=2t-4$. But this result does not hold for large $t$. In particular, 
$f(K_t) \in \Theta(t\sqrt{\ln t})$, where the lower bound is independently by \citet{fast4,fast5} and \citet{delaVega} (based on the work of \citet{BCE80}), and the upper bound is independently by \citet{fast4,fast5} and \citet{fast3}. Later, \citet{fast6} determined the exact asymptotic constant.

Subsequently, other authors have considered $f(H)$ for arbitrary graphs $H$. \citet{at-extreme} surveys some of these results. \citet{myersthom} determined an upper bound on $f(H)$ for all $H$, which is tight up to lower order terms when $H$ is dense. 
Hence much recent work has focused on $f(H)$ for sparse graphs $H$. There have been two major approaches in this case.

\subsection*{Specific Sparse Graphs} 

The first approach is to consider specific sparse graphs $H$. For example, \citet{k2t-myers} considered unbalanced complete bipartite graphs $K_{s,t}$ where $s \ll t$. It is easily seen that $f(K_{1,t})=t-1$. \citeauthor{k2t-myers} proved that $f(K_{2,t})=t+1$ for large $t$, and \citet{edk2t} proved the same result for all $t$. \citet{edk3t} proved that $f(K_{3,t})=t+3$ for large $t$. In fact, for all these results, the authors determined the exact maximum number of edges in a $K_{s,t}$-minor-free graph (for $s\leq 3$). 

\citeauthor{k2t-myers}  conjectured that $f(K_{s,t}) \leq c_st$ for some constant $c_s$ depending only on $s$. 
Strengthenings of this conjecture were independently proved by 
\citet{edkst-ko} and \citet{edkst}. In particular, \citet{edkst-ko} proved that for every $\epsilon \in (0,10^{-16})$, if $t$ is sufficiently large (with respect to $\epsilon$) and $s \leq \epsilon^{6}\tfrac{t}{\ln t}$, then
\begin{equation}
\label{KO}
f(K_{s,t}) \leq (1+\epsilon)t.
\end{equation}
\citet{edkst} proved a sharper bound on $f(K_{s,t})$ under a stronger assumption: if $t > (180s \log_2 s)^{1 + 6s \log_2 s}$ then 
\begin{equation}
\label{KP1}
t+3s - 5\sqrt{s} \leq f(K_{s,t}) \leq t + 3s.
\end{equation}
Later, \citet{kp-ii} proved an upper bound between  \eqref{KO} and \eqref{KP1} under a similar assumption to \eqref{KO}: if $s \leq \tfrac{t}{1000\log_2 t}$ then 
\begin{equation}
\label{KP2}
f(K_{s,t}) \leq t + 8s\log_2 s.
\end{equation}

\subsection*{General Sparse Graphs} 

A second approach for sparse graphs is to determine upper bounds on $f(H)$ in terms of invariants of $H$. This is the approach of a recent paper by \citet{superfast5}. Their main result is as follows: for every $t$-vertex graph $H$ with average degree $d(H)$ at least some constant $d_0$, then 
\begin{equation}
\label{RW1}
f(H) \leq 3.895\,t\,\sqrt{\ln d(H)},
\end{equation} 
If $d(H)$ is very small, then this result is not applicable. For all $H$, \citet{superfast5} proved that 
\begin{equation}
\label{RW2}
f(H) \leq (1+3.146\,d(H))\,t.
\end{equation} 
Inequalities \eqref{RW1} and \eqref{RW2} imply that for some constant $c$,  for every $t$-vertex graph $H$, 
\begin{equation}
\label{RW3}
f(H) \leq ct\sqrt{\ln(d(H)+2)}.
\end{equation} 
A lower bound of \citet{myersthom} shows that \eqref{RW3}  is tight for random or close-to-regular graphs $H$. But for some graphs  it is not tight. For example, for $K_{s,t}$ with $s\ll t$, inequality \eqref{RW3} says that $f(K_{s,t})\leq ct\sqrt{\ln s}$ since $d(K_{s,t}) \approx 2s$, whereas $f(K_{s,t})= \Theta(t)$ as discussed above. 

\subsection*{Our Results} 

In this paper, we make a first attempt at unifying these two approaches, both improving the bounds of \citet{superfast5} in certain cases, and generalising the above mentioned results on unbalanced complete bipartite graphs up to a constant factor.

\begin{theorem}
\label{larged}
There is a constant $d_0$ such that for  integers $s\geq 0$ and $t\geq 3$ with $s \leq 10^{-5} \tfrac{t}{\ln t}$, for every $(s+t)$-vertex graph $H$ and set $S$ of $s$ vertices in $H$ such that $d(H-S)\geq d_0$, 
$$f(H) \leq 3.895\,t\sqrt{\ln d(H-S)}.$$
\end{theorem}

\cref{larged} solves the second open problem of \citet{superfast5}, and is an improvement over \eqref{RW1} when $H$ contains a set $S$ of vertices with high degree (compared to the average degree). Then \cref{larged} forces $H$ as a minor in a graph $G$ as long as $d(G) > 3.895\,t\sqrt{\ln d}$, where $d$ is the average degree of $H-S$, instead of $H$ itself. Since vertices in $S$ have high degree, one expects that $d(H-S)$ is significantly less than $d(H)$. 

Our second contribution relates $f(H)$ directly to $f(H-S)$.

\begin{theorem}
\label{genfh}
For $\epsilon \in (0,1]$ and integers $s\geq 0$ and $t\geq 3$ with $s \leq \frac{\epsilon}{100} \tfrac{t}{\ln t}$, for every graph $H$ and set $S$ of $s$ vertices in $H$, 
$$f(H) \leq 4\ceil{(1+\epsilon)f(H-S)}.$$ 
%
%
\end{theorem}

\cref{genfh} may allow a better result that \cref{larged} if we happen to know a good upper bound on $f(H-S)$. It also makes no explicit assumption on $d(H-S)$. Theorems~\ref{larged} and \ref{genfh} together imply the following general result.

\begin{theorem}
\label{total}
There is a constant $c$ such that for  integers $s\geq 0$ and $t\geq 3$ with $s \leq 10^{-5} \tfrac{t}{\ln t}$, 
for every $(s+t)$-vertex graph $H$ and set $S$ of vertices in $H$ of size $s$,
$$f(H) \leq ct\sqrt{\ln(d(H-S)+2)}.$$
\end{theorem}

Consider \cref{total} with $H=K_{s,t}$, where $S$ is the smaller colour class. Thus $d(H-S)=0$ and  \cref{total} says that $f(H)\leq ct$, for some constant $c$ independent of $s$, which is within a constant factor of the bounds in  \eqref{KO}, \eqref{KP1} and \eqref{KP2}. Note however that these older results are still stronger, since \cref{total} has a large multiplicative constant. On the other hand, \cref{total} applies more generally when $H-S$ is not an independent set.

The following section contains a few preliminary lemmas. Section~\ref{sec:thm1} contains proofs of Theorems~\ref{larged}, \ref{genfh} and \ref{total}. Section~\ref{sec:thm2} contains an observation about $f(H)$ when $H$ is series-parallel. Section~\ref{sec:future} considers some future directions.

\section{Preliminaries}

We first prove a few useful preliminaries before proving our theorems.

A \emph{connected dominating set} $A$ of a graph $G$ is a set of vertices such that the induced subgraph $G[A]$ is connected, and each vertex of $V(G)$ is either in $A$ or adjacent to a vertex in $A$. \cref{cdset} is well known, and weaker than other previous results such as that by \citet{west_dom}. We present it here for completeness and because the upper bound on the order of the connected dominating set has no lower order terms, which simplifies the calculations in Section~\ref{sec:thm1}. Similarly, \cref{sblocks} resembles a previous result of \citet{edkst}.

\begin{lemma}
\label{cdset}
Every graph $G$ with $n$ vertices and minimum degree at least $\tfrac{1}{2}n$ has a connected dominating set of order less than $2\log_{2}n$.
\end{lemma}
\begin{proof}
Let $A$ be a set of $\floor{\log_{2}n}$ vertices in $G$ chosen uniformly at random. For each vertex $x$ of $G-A$, since $\deg_G(x)\geq\frac{n}{2}$, the probability that $x$ has no neighbour in $A$ is less than $(\tfrac{1}{2})^{\log_2 n} = \tfrac{1}{n}$, and so the probability that $A$ does not dominate $G$ is less than $1$. 
Hence there is some choice of $A$ that does dominate $G$. Label the vertices of this $A$ by $x_1, \dots, x_{\floor{\log_2n}}$. For $i=1,\dots,\floor{\log_2n}-1$, let $v_i = x_{i+1}$ if $x_i$ and $x_{i+1}$ are adjacent, otherwise let $v_i$ be a common neighbour of $x_i$ and $x_{i+1}$, which exists since $x_i$ and $x_{i+1}$ each have at least $\tfrac{n}{2}$ neighbours in a set of $n-2$ vertices. Then $A \cup \{v_1, \dots, v_{\floor{\log_2n}-1}\}$ is a connected dominating set of $G$ with less than $2\log_2n$ vertices. 
\end{proof}

\begin{lemma}
\label{sblocks}
For every integer $s \geq 1$ and $\rho \geq \tfrac{1}{2}$, every graph $G$ with $n$ vertices and minimum degree at least $\rho n + 2s\log_2n$ contains vertex sets $A_1,\dots,A_s$ and subgraphs $G_0,G_1,\dots,G_s$ such that $G_0 = G$, $G_i = G-(A_1 \cup \dots \cup A_{i})$ for $i \in \{1,\dots,s\}$, and
\begin{enumerate}
\item[(a)] $A_i$ is a connected dominating set in $G_{i-1}$,
\item[(b)] $|A_i| < 2\log_2n$,
\item[(c)] $G_i$ has minimum degree at least $\rho n + 2(s-i)\log_2n$.
\item[(d)] $G_i$ has at least $n-2i\log_2n$ vertices.
\end{enumerate}
\end{lemma}
\begin{proof}
We use induction on $i$. Say $i=1$. Let $A_1$ be a connected dominating set in $G_0=G$ from \cref{cdset}. This satisfies $(a)$ and $(b)$. The minimum degree of $G_1 = G - A_1$ is at least the minimum degree of $G$ minus $|A_1|$, which proves $(c)$. The number of vertices of $G_1$ is at least $n-|A_1|$, which proves $(d)$.

Assume our lemma holds for $i-1$. Let $A_i$ be a connected dominating set in $G_{i-1}$ from \cref{cdset}; such a set exists since $G_{i-1}$ has minimum degree greater than $\rho n \geq \tfrac{1}{2}n$. This satisfies $(a)$ and $(b)$. Finally, $G_i$ has minimum degree at least $\rho n + 2(s-(i-1))\log_2n - |A_i| > \rho n - 2(s-i)\log_2n$ and $|V(G_i)| \geq n-2(i-1)\log_2n - |A_i| > n -2i\log_2n$, as required.
\end{proof}

Finally, we cite two key results of \citet{superfast5}.

\begin{lemma}[\citet{superfast5}, Lemma 2.5]
\label{rw1}
For every integer $k \geq 1$, every graph with average degree at least $4k$ contains a complete graph $K_k$ as a minor or contains a minor with $n$ vertices and minimum degree $\delta$, where $\delta \geq 0.6518n$, and $k \leq \delta < n \leq 4k$.
\end{lemma}

\begin{lemma}[\citet{superfast5}, Lemma 5.1]
\label{rw2}
For all $\lambda \in (\tfrac{1}{2},1)$ and $\epsilon \in (0, \lambda)$ there exists $d_0(\lambda,\epsilon)$ such that for every graph $H$ with $t$ vertices and average degree $d \geq d_0$ every graph $G$ with $n \geq (1+\epsilon)\ceil{\sqrt{\log_bd}}\,t$ vertices and minimum degree at least $\lambda n$ contains $H$ as a minor, where $b = (1-\lambda + \epsilon)^{-1}$.
\end{lemma}

\section{Proofs of Theorems}
\label{sec:thm1}

\begin{proof}[Proof of \cref{larged}]
Let $G$ be a graph with $d(G) \geq 3.895\,t\sqrt{\ln d}$ where $d:=d(H-S)\geq d_0$. Let $k:= \floor{\tfrac{1}{4}(3.895\,t\sqrt{\ln d})}$. Our goal is to show that $G$ contains an $H$-minor. We can take $d_0$ large enough so that $k \geq 2t > 1$. By \cref{rw1}, $G$ contains either $K_{2t}$ or $G'$ as a minor, where $G'$ is a graph with $n$ vertices and minimum degree at least $0.6518n$ such that $k+1 \leq n \leq 3.895\,t\sqrt{\ln d}$. If $G$ contains a $K_{2t}$ minor, then $G$ contain a $K_{s+t}$ minor and we are done. Otherwise apply \cref{sblocks} to $G'$ where $\rho = 0.6517$. Since $s \leq 10^{-5}\tfrac{t}{\ln t} < 5 \times 10^{-5}\tfrac{t}{\log_2 t} \leq 5 \times 10^{-5}\tfrac{n}{\log_2n}$, we have $2s\log_2 n \leq 10^{-4}n$, and it follows that $G'$ has minimum degree at least $\rho n + 2s\log_2 n$, as required. Let $G'':=G_s$ from \cref{sblocks}. Then $G''$ has minimum degree at least $0.6517n > 0.6517|V(G'')|$. 

We wish to find an $(H-S)$-minor of $G''$. Then contracting each $A_i$ to a single vertex gives an $H$-minor in $G'$. 
We now verify that \cref{rw2} gives the desired $(H-S)$-minor in $G''$, where $\lambda := 0.6517$ and $\epsilon := 10^{-6}$ and $b:=(1-\lambda+\epsilon)^{-1}$. Clearly $G''$ has sufficient minimum degree (by our choice of $\lambda$) and the average degree $d$ of $H-S$ is sufficiently large (since we may assume $d_0$ is sufficiently large in terms of absolute constants), and so all that remains is to ensure that $G''$ has sufficiently many vertices. Note $|V(G'')| \geq n - 2s\log_2n$ from \cref{sblocks}(d). We may choose $d_0$ large enough so that $\tfrac{3.895\,t\sqrt{\ln d}}{\log_2(3.895\,t\sqrt{\ln d})} \geq 100 \tfrac{t}{\ln t}$ and so $s \leq 10^{-5} \tfrac{t}{\ln t} \leq 10^{-7}\tfrac{3.895\,t\sqrt{\ln d}}{\log_2(3.895\,t\sqrt{\ln d})}$. Thus it follows that 
\begin{align*}
|V(G'')| &\geq n - 2(10^{-7})\tfrac{3.895\,t\sqrt{\ln d}}{\log_2 (3.895\,t\sqrt{\ln d})}\log_2 n \\
&\geq \floor{\tfrac{1}{4}(3.895\,t\sqrt{\ln d})}+1 - 2(10^{-7})(3.895\,t\sqrt{\ln d}) \\ 
&\geq \tfrac{1}{4}(3.895\,t\sqrt{\ln d}) - 2(10^{-7})(3.895\,t\sqrt{\ln d}) \\ 
&= (\tfrac{1}{4}-2(10^{-7}))(3.895\,t\sqrt{\ln d}) \\ 
&= (\tfrac{1}{4}-2(10^{-7}))(3.895\,t\sqrt{\ln b}\sqrt{\log_b d}) \\ 
&\geq 1.00002\sqrt{\log_b d}\,t \\ 
&= (1+20\epsilon)\sqrt{\log_b d}\,t \\
 &\geq (1+\epsilon)(1+\sqrt{\log_bd})t \text{\qquad(taking $d_0$ large enough)}\\ 
 &\geq (1+\epsilon)\ceil{\sqrt{\log_b d}}\,t.
\end{align*} 
Hence it follows that $G''$ contains an $(H-S)$-minor. This is an $(H-S)$-minor in $G'$ that avoids the sets $A_1,\dots,A_s$. Hence $G'$ (and also $G$) contains our desired $H$-minor.
\end{proof}

\begin{proof}[Proof of \cref{genfh}]
Let $G$ be a graph with average degree $d(G)\geq 4k$, where $k:=\ceil{(1+\epsilon)f(H-S)}$. If $s=0$ this result is trivial, so assume $s \geq 1$. It follows from \cref{rw1} that $G$ contains either a $K_k$-minor or a minor $G'$ with $n$ vertices and minimum degree $\delta(G') \geq 0.6518n$ such that $k \leq \delta(G') < n \leq 4k$.  Since $f(H-S) \geq t-2$ (as $K_{t-1}$ has no $(H-S)$-minor), it follows that $k \geq (1+\epsilon)f(H-S) \geq (1+\epsilon)(t-2) \geq t-2 + 100s\ln{t}-2\epsilon \geq t + 100s - 4 > t+s$. Hence a $K_k$-minor contains an $H$-minor. Now assume \cref{rw1} finds $G'$ as a minor.

We wish to apply \cref{sblocks} to $G'$ with $\rho=\tfrac{1}{2}$. It is sufficient to show that $0.1518n \geq 2s\log_2 n$. Since $n \geq k+1 \geq (1+\epsilon)f(H-S) + 1 \geq t-1 + \epsilon f(H-S)$, and $n$ and $t$ are integers, it follows that $n \geq t$. Thus $s \leq \tfrac{\epsilon}{100}\tfrac{t}{\ln t}\leq \tfrac{\epsilon}{100}\tfrac{n}{\ln n}$ and so $2s\log_2 n \leq \tfrac{n}{50\ln 2} < 0.1518n$ as required. Let $G'':=G_s$ from \cref{sblocks}. By \cref{sblocks}(b), it follows that $G''$ has minimum degree at least $\delta(G') - 2s\log_2n \geq k - 2s\log_2n$. From our upper bound on $s$ and since $k < n \leq 4k$, it follows that $k - 2s\log_2n \geq k - \epsilon\tfrac{n}{50\ln n}\log_2n \geq k(1 - \tfrac{2}{25\ln 2}\epsilon ) \geq (1+\epsilon)(1-0.1155\epsilon)f(H-S)>f(H-S)$. Thus $d(G'')>f(H-S)$ and $G''$ contains an $(H-S)$-minor. Contracting the sets $A_1,\dots,A_s$ to single vertices gives an $H$-minor in $G'$, and thus also in $G$.
\end{proof}

\begin{proof}[Proof of \cref{total}]
First suppose that $d(H-S) \geq d_0$, where $d_0$ is from \cref{larged}. Let $c \geq 3.895$. By \cref{larged}, $f(H) \leq 3.895\,t\sqrt{\ln d(H-S)} \leq ct\sqrt{\ln (d(H-S)+2)}$ as required. Alternatively, $d(H-S) < d_0$. By \eqref{RW2} we have $f(H-S) \leq (1+3.146d(H-S))t \leq (1+3.146d_0)t$. Let $\epsilon = 10^{-3}$. It follows from \cref{genfh} that $f(H) \leq 4\ceil{(1+\epsilon)(1+3.146d_0)t}$. Setting $c$ large enough, this proves $f(H) \leq ct\sqrt{\ln (d(H-S)+2)}$, as required. 
\end{proof}

\section{Series-Parallel Graphs}
\label{sec:thm2}

A graph is \emph{series-parallel} if it contains no $K_4$ minor (or equivalently, it has treewidth at most 2). \citet[Lemma~3.3]{superfast5} proved that $f(H) \leq 6.929t$ for every $t$-vertex 2-degenerate graph $H$. Every series-parallel graph $H$ is 2-degenerate, implying $f(H)\leq 6.929t$. Here we make the following improvement. 

\begin{proposition}
\label{2tree} For every $t$-vertex series-parallel graph $H$, $$f(H) \leq 2t-4.$$
\end{proposition}

\begin{proof}
A \emph{$2$-tree} is a graph that can be constructed by starting with $K_3$ and repeatedly selecting an edge and adding a new vertex adjacent to exactly the endpoints of that edge. It is well known that 2-trees are exactly the edge-maximal series-parallel graphs \cite{D05}. Hence it suffices to prove that $f(H) \leq 2t-4$ for every $t$-vertex 2-tree  $H$. Say a graph $G$ is \emph{minor-minimal} (with respect to $t$) if $d(G) \geq 2t-4$ but every proper minor of $G$ has average degree less than $2t-4$. Considering the effect of contracting an edge on the average degree, it is easily seen that every edge of a minor-minimal graph is in at least $t-2$ triangles; see \cite[Lemma 2.1]{superfast5}. It suffices to prove  that every minor-minimal graph $G$ contains an $H$-minor. In fact, we prove that every $2$-tree $H_0$ on at most $t$ vertices can be embedded in 
$G$ as a subgraph. 

We do this by induction on $|V(H_0)|$ with $G$ fixed. Suppose $|V(H_0)|=3$. Then every edge of $G$ is in at least $t-2 \geq 1$ triangles, and it is trivial to embed $H_0$. Now suppose $|V(H_0)|=i>3$. There is a vertex $v \in V(H_0)$ with neighbours $x$ and $y$ such that $H_0 - v$ is also a $2$-tree. By  induction, $H_0 - v$ can be embedded in $G$. Let $x',y'$ denote the vertices of $G$ where $x,y$ are respectively embedded. Then $x'y'$ is an edge of $G$ in at least $t-2 \geq |V(H_0) - \{v,x,y\}| +1$ triangles. Hence there exists a common neighbour $w$ of $x'$ and $y'$ in $G$ where no vertex of $H_0 - \{v,x,y\}$ is embedded. Clearly, neither $x$ nor $y$ are embedded at $w$. Embedding $v$ at $w$, we obtain $H_0$ as a subgraph of $G$, as required.
\end{proof}

\section{Open Problems}
\label{sec:future}

An obvious question is whether the upper bound on $f(H)$ provided by \cref{total} is within a constant factor of optimal for all $H$. If $S$ is a small (perhaps empty) set of vertices in $H$ and $H-S$ is sufficiently large and either random or close-to-regular, then it follows from the lower bound of \citet{myersthom} that $$f(H) \geq f(H-S) \geq c|V(H-S)|\sqrt{\ln(d(H-S))},$$ and  \cref{total} is within a constant factor of optimal.  

Is it possible that \cref{total} always gives a result within a constant factor of optimal (for the best possible choice of $S$)? If not, for which graphs is \cref{total} not within a constant factor of optimal, and what approach should we take for such graphs? Is it possible to determine a small set of ``techniques" for upper bounding $f(H)$ such that for each graph $H$, one of these bounds is within a constant factor of optimal? Similarly, is there a constant factor approximation algorithm for computing $f(H)$?
 

\begin{thebibliography}{22}
\providecommand{\natexlab}[1]{#1}
\providecommand{\url}[1]{\texttt{#1}}
\providecommand{\urlprefix}{}
\expandafter\ifx\csname urlstyle\endcsname\relax
  \providecommand{\doi}[1]{doi:\discretionary{}{}{}#1}\else
  \providecommand{\doi}{doi:\discretionary{}{}{}\begingroup
  \urlstyle{rm}\Url}\fi

\bibitem[{Bollob{\'a}s et~al.(1980)Bollob{\'a}s, Catlin, and
  Erd{\H{o}}s}]{BCE80}
\textsc{B{\'e}la Bollob{\'a}s, Paul~A. Catlin, and Paul Erd{\H{o}}s}.
\newblock Hadwiger's conjecture is true for almost every graph.
\newblock \emph{European J. Combin.}, 1(3):195--199, 1980.
\newblock \doi{10.1016/S0195-6698(80)80001-1}.
\newblock \msn{593989}.

\bibitem[{Caro et~al.(2000)Caro, West, and Yuster}]{west_dom}
\textsc{Yair Caro, Douglas~B. West, and Raphael Yuster}.
\newblock Connected domination and spanning trees with many leaves.
\newblock \emph{SIAM J. Discrete Math.}, 13(2):202--211, 2000.
\newblock \doi{10.1137/S0895480199353780}.

\bibitem[{Chudnovsky et~al.(2011)Chudnovsky, Reed, and Seymour}]{edk2t}
\textsc{Maria Chudnovsky, Bruce~A. Reed, and Paul Seymour}.
\newblock The edge-density for {$K_{2,t}$} minors.
\newblock \emph{J. Combin. Theory Ser. B}, 101(1):18--46, 2011.
\newblock \doi{10.1016/j.jctb.2010.09.001}.

\bibitem[{de~la Vega(1983)}]{delaVega}
\textsc{W.~Fernandez de~la Vega}.
\newblock On the maximum density of graphs which have no subcontraction to
  {$K^{s}$}.
\newblock \emph{Discrete Math.}, 46(1):109--110, 1983.
\newblock \doi{10.1016/0012-365X(83)90280-7}.
\newblock \msn{0708172}.

\bibitem[{Diestel(2010)}]{D05}
\textsc{Reinhard Diestel}.
\newblock \emph{Graph theory}, vol. 173 of \emph{Graduate Texts in
  Mathematics}.
\newblock Springer-Verlag, Berlin, 4th edn., 2010.
\newblock \urlprefix\url{http://diestel-graph-theory.com/}.

\bibitem[{J{\o}rgensen(1994)}]{jorg}
\textsc{Leif~K. J{\o}rgensen}.
\newblock Contractions to {$K_8$}.
\newblock \emph{J. Graph Theory}, 18(5):431--448, 1994.
\newblock \doi{10.1002/jgt.3190180502}.

\bibitem[{Kostochka(1982)}]{fast4}
\textsc{Alexandr~V. Kostochka}.
\newblock The minimum {H}adwiger number for graphs with a given mean degree of
  vertices.
\newblock \emph{Metody Diskret. Analiz.}, 38(38):37--58, 1982.

\bibitem[{Kostochka(1984)}]{fast5}
\textsc{Alexandr~V. Kostochka}.
\newblock Lower bound of the {H}adwiger number of graphs by their average
  degree.
\newblock \emph{Combinatorica}, 4(4):307--316, 1984.
\newblock \doi{10.1007/BF02579141}.

\bibitem[{Kostochka and Prince(2008)}]{edkst}
\textsc{Alexandr~V. Kostochka and Noah Prince}.
\newblock On {$K_{s,t}$}-minors in graphs with given average degree.
\newblock \emph{Discrete Math.}, 308(19):4435--4445, 2008.
\newblock \doi{10.1016/j.disc.2007.08.041}.

\bibitem[{Kostochka and Prince(2010)}]{edk3t}
\textsc{Alexandr~V. Kostochka and Noah Prince}.
\newblock Dense graphs have {$K_{3,t}$} minors.
\newblock \emph{Discrete Math.}, 310(20):2637--2654, 2010.
\newblock \doi{10.1016/j.disc.2010.03.026}.

\bibitem[{Kostochka and Prince(2012)}]{kp-ii}
\textsc{Alexandr~V. Kostochka and Noah Prince}.
\newblock On {$K_{s,t}$}-minors in graphs with given average degree, {II}.
\newblock \emph{Discrete Math.}, 312(24):3517--3522, 2012.
\newblock \doi{10.1016/j.disc.2012.08.004}.

\bibitem[{K{\"u}hn and Osthus(2005)}]{edkst-ko}
\textsc{Daniela K{\"u}hn and Deryk Osthus}.
\newblock Forcing unbalanced complete bipartite minors.
\newblock \emph{European J. Combin.}, 26(1):75--81, 2005.
\newblock \doi{10.1016/j.ejc.2004.02.002}.

\bibitem[{Mader(1967)}]{fast1}
\textsc{Wolfgang Mader}.
\newblock Homomorphieeigenschaften und mittlere {K}antendichte von {G}raphen.
\newblock \emph{Mathematische Annalen}, 174:265--268, 1967.
\newblock \doi{10.1007/BF01364272}.

\bibitem[{Mader(1968)}]{fast2}
\textsc{Wolfgang Mader}.
\newblock Homomorphies\"atze f\"ur {G}raphen.
\newblock \emph{Math. Ann.}, 178:154--168, 1968.
\newblock \doi{10.1007/BF01350657}.
\newblock \msn{0229550}.

\bibitem[{Myers(2003)}]{k2t-myers}
\textsc{Joseph~Samuel Myers}.
\newblock The extremal function for unbalanced bipartite minors.
\newblock \emph{Discrete Math.}, 271(1-3):209--222, 2003.
\newblock \doi{10.1016/S0012-365X(03)00051-7}.

\bibitem[{Myers and Thomason(2005)}]{myersthom}
\textsc{Joseph~Samuel Myers and Andrew Thomason}.
\newblock The extremal function for noncomplete minors.
\newblock \emph{Combinatorica}, 25(6):725--753, 2005.
\newblock \doi{10.1007/s00493-005-0044-0}.

\bibitem[{Reed and Wood(2015)}]{superfast5}
\textsc{Bruce~A. Reed and David~R. Wood}.
\newblock Forcing a sparse minor.
\newblock \emph{Combin. Probab. Comput.}, 2015.
\newblock \doi{10.1017/S0963548315000073}.

\bibitem[{Seymour(2015)}]{Seymour-HC}
\textsc{Paul~D. Seymour}.
\newblock Hadwiger's conjecture, 2015.
\newblock
  \urlprefix\url{https://web.math.princeton.edu/~pds/papers/hadwiger/paper.pdf}.

\bibitem[{Song and Thomas(2006)}]{SongT}
\textsc{Zi{-}Xia Song and Robin Thomas}.
\newblock The extremal function for {$K_9$} minors.
\newblock \emph{J. Combin. Theory Ser. B}, 96(2):240--252, 2006.
\newblock \doi{10.1016/j.jctb.2005.07.008}.

\bibitem[{Thomason(1984)}]{fast3}
\textsc{Andrew Thomason}.
\newblock An extremal function for contractions of graphs.
\newblock \emph{Math. Proc. Cambridge Philos. Soc.}, 95(2):261--265, 1984.
\newblock \doi{10.1017/S0305004100061521}.

\bibitem[{Thomason(2001)}]{fast6}
\textsc{Andrew Thomason}.
\newblock The extremal function for complete minors.
\newblock \emph{J. Combin. Theory Ser. B}, 81(2):318--338, 2001.
\newblock \doi{10.1006/jctb.2000.2013}.

\bibitem[{Thomason(2006)}]{at-extreme}
\textsc{Andrew Thomason}.
\newblock Extremal functions for graph minors.
\newblock In \emph{More Sets, Graphs and Numbers}, vol.~15 of \emph{Bolyai
  Society Mathematical Studies}, pp. 359--380. 2006.
\newblock \doi{10.1007/978-3-540-32439-3\_17}.

\end{thebibliography}

\end{document}